\documentclass[11pt]{amsart}

\usepackage[a4paper,margin=1.05in]{geometry}
\usepackage[T1]{fontenc}
\usepackage[utf8]{inputenc}
\usepackage{lmodern}
\usepackage{amsmath,amssymb,amsthm,mathtools}
\usepackage{microtype}
\usepackage{enumitem}
\usepackage{graphicx}
\usepackage{booktabs}
\usepackage{url}
\usepackage{hyperref}
\usepackage{float}

\hypersetup{
  colorlinks=true,
  linkcolor=blue,
  citecolor=blue,
  urlcolor=blue,
  hypertexnames=false
}

\newcommand{\codeRepo}{\url{https://github.com/tomaszkania/gauss-lobatto-bounds}}
\newcommand{\codeVersion}{v1.1.0}

\newtheorem{theorem}{Theorem}[section]
\newtheorem{proposition}[theorem]{Proposition}
\newtheorem{lemma}[theorem]{Lemma}
\newtheorem{corollary}[theorem]{Corollary}
\theoremstyle{definition}
\newtheorem{definition}[theorem]{Definition}
\theoremstyle{remark}
\newtheorem{remark}[theorem]{Remark}

\numberwithin{equation}{section}

\newcommand{\R}{\mathbb{R}}
\newcommand{\I}{\mathcal{I}}

\newcommand{\dd}{\,\mathrm{d}}
\newcommand{\abs}[1]{\left|#1\right|}
\newcommand{\bracks}[1]{\left[#1\right]}

\renewcommand{\le}{\leqslant}
\renewcommand{\ge}{\geqslant}

\title{Refined Gauss--Lobatto bounds for odd-order convexity}

\author[T.~Kania]{Tomasz Kania}
\address[T.~Kania]{Mathematical Institute\\Czech Academy of Sciences\\v{Z}itn\'a 25 \\115 67 Praha 1\\Czech Republic and Institute of Mathematics and Computer Science\\ Jagiellonian University\\ {\L}ojasiewicza 6, 30-348 Krak\'{o}w, Poland}
\email{kania@math.cas.cz, tomasz.marcin.kania@gmail.com}

\author[Sz.~W\k{a}sowicz]{Szymon W\k{a}sowicz}
\address[Sz.~W\k{a}sowicz]{Department of Mathematics\\
 University of Bielsko--Bia\l{}a\\
 Willowa 2, 43--309 Bielsko--Bia\l{}a, Poland}
\email{swasowicz@ubb.edu.pl} 

\thanks{IM CAS (RVO 67985840). }

\begin{document}

\begin{abstract}
For functions that are convex of higher order, classical extremalities identify Gaussian quadrature
as a lower bound and Gauss--Lobatto quadrature as an upper bound for the integral.
We sharpen this bracket in every odd order: if $f$ is $(2n-1)$-convex on $[-1,1]$, then the exact integral lies
not merely between the $n$-point Gauss--Legendre rule $G_n[f]$ and the $(n+1)$-point Gauss--Lobatto rule $L_{n+1}[f]$,
but already between $G_n[f]$ and their midpoint.
This yields the certified estimator
\[
Q_n[f] := \frac{3}{4}G_n[f]+\frac{1}{4}L_{n+1}[f],
\qquad
\abs{\I[f]-Q_n[f]}\le \frac{1}{4}\abs{L_{n+1}[f]-G_n[f]},
\]
and the constant $\tfrac14$ is sharp.
We also record a complementary phenomenon in even order:
for $2n$-convex functions the Gauss--Radau endpoint rules still bracket the integral,
but the midpoint of the two Radau rules cannot yield a one-sided refinement.
The obstruction is explained by an odd, sign-changing Radau quadrature kernel.
Finally, we distinguish numerical certification from ordinary high-order approximation:
we include concrete examples from spline/moment/statistical models and report experiments for truncated powers and near-pole Stieltjes kernels,
where the shape certificate remains meaningful when classical smoothness-based error constants are unavailable or severely pessimistic.
\end{abstract}

\maketitle

\section{Introduction}

Higher-order convexity is useful precisely because it combines a qualitative shape assumption with quantitative consequences.
It occurs in several concrete settings.
In spline approximation, truncated powers
\((x-\tau)_+^r\) form the classical truncated-power basis for polynomial splines, and integrals of such functions enter mass, load, smoothing, and normalisation terms; see, for example, de~Boor~\cite{deBoor2001}.
In moment problems and the theory of Tchebycheff systems, the same divided-difference language is natural and kernels such as \((x+\lambda)^{-1}\) or Stieltjes--Markov transforms have alternating derivatives of all orders; see Karlin--Studden~\cite{KarlinStudden1966}.
In shape-constrained statistics, convex regression and density estimation, and more generally $k$-monotone density estimation, lead to spline estimators and integrated spline processes; see Groeneboom--Jongbloed--Wellner~\cite{GroeneboomJongbloedWellner2001} and Balabdaoui--Wellner~\cite{BalabdaouiWellner2007}.
Thus the relevant integrands are not always analytic test functions; they may be piecewise polynomial, have a high-order kink, or possess a nearby pole outside the interval while retaining a fixed higher-order shape.

In numerical integration, a shape assumption can be turned into a rigorous stopping rule.
This point of view is classical in the Peano-kernel theory of definite quadrature formulae; see the survey of F\"orster~\cite{Foerster1993}.
For the unweighted integral
\[
\I[f]=\int_{-1}^1 f(x)\,\mathrm{d}x,
\]
the starting point is the classical extremality of Gauss and Gauss--Lobatto quadratures on the cone of $(2n-1)$-convex functions:
\[
G_n[f]\le \I[f]\le L_{n+1}[f].
\]
For $n=1$, this is the Hermite--Hadamard inequality, with $G_1$ the midpoint rule and $L_2$ the trapezoidal rule.
For higher $n$, the bounds are less familiar but remarkably rigid: they hold for an entire cone of integrands and not for one smoothness scale chosen in advance.

This paper strengthens the bracket in every odd order.
The integral is always closer to the Gauss side:
\[
G_n[f]\le \I[f]\le \frac{G_n[f]+L_{n+1}[f]}{2}.
\]
Taking the centre of this refined interval yields
\[
Q_n[f]=\frac34G_n[f]+\frac14L_{n+1}[f],
\qquad
\abs{\I[f]-Q_n[f]}\le \frac14\abs{L_{n+1}[f]-G_n[f]}.
\]
The constant $1/4$ is sharp for this estimator.

We also explain why the analogous hope fails in even order.
There, the natural bracketing rules are the two Gauss--Radau endpoint formulas.
Although they still trap the integral for $2n$-convex functions, their midpoint does not consistently lie above or below the integral.
We make this failure precise by studying the associated Peano kernel, which is odd and sign-changing.
This is the kernel version of the familiar obstruction that a definite quadrature formula of odd Peano order cannot be symmetric.

For analytic functions on short intervals, ordinary high-order Gauss--Legendre quadrature can of course be extremely efficient when one measures actual error against a known exact value.
The numerical discussion below therefore separates approximation from certification.
The smooth tests $1/x$, $e^x$, and $\log(1+x)$ are retained only as sanity checks in the companion repository.
The paper itself focuses on the intended use case: certified integration under a higher-order shape constraint, especially for truncated powers and near-pole Stieltjes kernels where classical smoothness-based Gauss error constants are unavailable or very pessimistic.

From a computational perspective, Gauss--Lobatto nodes and weights are standard ingredients in spectral methods and collocation schemes, and stable procedures for their computation are well developed; see Gautschi's work on high-order Gauss--Lobatto and generalised Gauss--Radau/Lobatto rules \cite{Gautschi2000,Gautschi2004}, as well as applications in optimal control and fractional problems \cite{Williams2006,EsmaeiliMilovanovic2014}.
Our focus is different: we treat these rules as certification devices for integrals of higher-order convex functions.
The refined bracket hinges on a pointwise dominance of Peano kernels for the Gauss/Gauss--Lobatto pair.
For Clenshaw--Curtis-type rules based on Chebyshev points, or for Szeg\H{o}--Lobatto rules on the unit circle \cite{JagelsReichel2007}, we are not aware of an analogous dominance principle.

\paragraph{\emph{Organisation}.}
Each section begins with a short road map.  Section~\ref{sec:hoc} recalls higher-order convexity and a representation theorem.
Section~\ref{sec:quadrature} fixes the quadrature rules.
Section~\ref{sec:odd} contains the refined Gauss--Lobatto inequality and its consequences.
Section~\ref{sec:estimators} develops the certified estimator $Q_n$, proves sharpness, and records the optimality of the interval centre.
Section~\ref{sec:smooth} discusses the smooth/analytic regime and a degree-raising combination.
Section~\ref{sec:even} introduces the Radau quadrature kernel and proves the even-order obstruction.
Section~\ref{sec:numerics} presents numerical experiments and points to a companion notebook.

\section{Higher-order convexity}\label{sec:hoc}

Higher-order convexity can be defined in several equivalent ways; we choose a version that interacts cleanly with Peano kernels.
The key point is that the cone of higher-order convex functions is generated by truncated power functions plus polynomials,
which is exactly what we need for quadrature comparisons.

\subsection{Definitions and basic properties}

To keep the exposition concrete, we work on $[-1,1]$ and the unweighted integral, but the affine invariance of divided differences
means that all statements translate immediately to a general interval $[a,b]$. We write $\Pi_k$ for the space of real polynomials of degree at most $k$.

\begin{definition}[Higher-order convexity]\label{def:higherconv}
Let $m\in\mathbb{N}$ and let $f:[-1,1]\to\R$.
We say that $f$ is \emph{$(m-1)$-convex} if all divided differences of order $m$ are non-negative, that is,
\[
[x_0,\dots,x_m;f]\ge 0
\quad\text{for all distinct }x_0,\dots,x_m\in[-1,1].
\]
If all such divided differences are nonpositive, we call $f$ \emph{$(m-1)$-concave}.
\end{definition}

\begin{remark}
This recovers ordinary convexity for $m=2$.
If $f\in C^m([-1,1])$, then $f$ is $(m-1)$-convex whenever $f^{(m)}\ge 0$ on $[-1,1]$
(a classical consequence of the mean-value form of divided differences).
\end{remark}

\begin{remark}
For later use, and as a reminder, we recall an equivalent recursive description.
A~function $f\colon I\to\mathbb{R}$ is $(m-1)$-convex on an interval $I$ if and only if
for every fixed $x_0\in I$ the function
\[
x \longmapsto [x,x_0,\dots,x_0;f]
\]
(the divided difference with $m$ arguments, where $x_0$ is repeated $m-1$ times)
is non-decreasing on $I\setminus\{x_0\}$.
Equivalently, $f$ is $(m-1)$-convex if and only if its divided differences of order
$m-1$ are convex functions in each variable. In particular, if $f\in C^{m}(I)$, this condition is equivalent to the pointwise
inequality $f^{(m)}\ge 0$ on $I$.
\end{remark}

The notion of higher-order convexity originates in the classical work of Hopf \cite{Hopf1926} and
Popoviciu \cite{Popoviciu1944}.
Popoviciu systematically introduced convexity of arbitrary order via divided
differences and established its fundamental properties, while Hopf studied related
monotonicity phenomena for higher derivatives.
Modern treatments of higher-order convexity are all rooted in these foundational
ideas.

\subsection{Popoviciu-type representation}

The next result is the technical bridge between higher-order convexity and Peano kernels.
We include it with a proof because it is the mechanism that allows us to test inequalities on a one-parameter family
of extremal functions.

\begin{theorem}[Representation by truncated powers]\label{thm:popoviciu}
Let $m\in\mathbb{N}$ and let $f:[-1,1]\to\R$ be $(m-1)$-convex.
Then there exist a polynomial $p\in\Pi_{m-1}$ and a non-negative Borel measure $\mu$ on $[-1,1]$ such that
\begin{equation}\label{eq:popoviciu}
f(x) = p(x) + \int_{-1}^1 (x-t)_+^{m-1}\,\mu(\dd t),
\qquad x\in[-1,1],
\end{equation}
where $(x-t)_+ := \max\{x-t,0\}$.
If $f$ is $(m-1)$-concave, the same holds with $\mu$ nonpositive.
\end{theorem}

\begin{proof}
We first treat the smooth case, which already contains the essential idea.
Assume $f\in C^m([-1,1])$ and $f^{(m)}\ge 0$.
Define
\[
\mu(\dd t) := \frac{f^{(m)}(t)}{(m-1)!}\,\dd t,
\qquad
p(x) := \sum_{k=0}^{m-1}\frac{f^{(k)}(-1)}{k!}(x+1)^k.
\]
Repeated integration of $f^{(m)}$ shows that, for every $x\in[-1,1]$,
\[
f(x)-p(x)
= \int_{-1}^x \frac{(x-t)^{m-1}}{(m-1)!}f^{(m)}(t)\,\dd t
= \int_{-1}^1 (x-t)_+^{m-1}\,\mu(\dd t),
\]
which is \eqref{eq:popoviciu}.

For a general $(m-1)$-convex $f$ (not assumed smooth), we use the standard distributional form of the Popoviciu--Boas--Widder representation.
In this form, non-negativity of divided differences of order $m$ is equivalent to the existence of a non-negative Borel measure $\nu$ such that
\[
D^m f=\nu
\]
in the sense of distributions on $(-1,1)$; see Popoviciu~\cite{Popoviciu1944}, Boas--Widder~\cite{BoasWidder1940}, and the integral characterisation of higher-order convexity in Bessenyei--P\'ales~\cite{BessenyeiPales2010}.
Define
\[
u_0(x):=\int_{-1}^1\frac{(x-t)_+^{m-1}}{(m-1)!}\,\nu(\dd t).
\]
Then $D^m u_0=\nu$, because $(x-t)_+^{m-1}/(m-1)!$ is the one-sided fundamental solution of the $m$-th derivative.
Consequently $D^m(f-u_0)=0$, and therefore $f-u_0$ agrees on $[-1,1]$ with a polynomial $p\in\Pi_{m-1}$.
Writing $\mu:=\nu/(m-1)!$ gives exactly \eqref{eq:popoviciu}.
The concave case follows by applying the convex case to $-f$.\end{proof}

\begin{remark}\label{rem:reduction}
The representation \eqref{eq:popoviciu} implies a reduction principle:
if a linear functional $\Lambda$ annihilates $\Pi_{m-1}$ and satisfies
$\Lambda[(\cdot-t)_+^{m-1}]\ge 0$ for all $t\in[-1,1]$, then $\Lambda[f]\ge 0$ for every $(m-1)$-convex $f$.
We will use this repeatedly with $\Lambda$ equal to differences of quadrature operators.
\end{remark}

\section{Gauss-type quadratures on \texorpdfstring{$[-1,1]$}{[-1,1]}}\label{sec:quadrature}

This section fixes notation for the quadrature rules and records the few structural facts we will need.
The main reason to be explicit is that later we will compare rules via their Peano kernels.

\subsection{The integral operator and three families of rules}

We write
\[
\I[f] := \int_{-1}^1 f(x)\,\dd x.
\]

\paragraph{\bf Gauss--Legendre.}
For $n\ge 1$, the $n$-point Gauss--Legendre rule is
\[
G_n[f] := \sum_{i=1}^n w_i\, f(x_i),
\]
where $x_i$ are the zeros of the Legendre polynomial $P_n$ and $w_i>0$ are the associated Christoffel numbers.
It is exact on $\Pi_{2n-1}$.

\paragraph{\bf Gauss--Lobatto.}
For $n\ge 1$, the $(n+1)$-point Gauss--Lobatto rule is
\[
L_{n+1}[f] := \sum_{i=0}^n v_i\, f(y_i),
\qquad
-1=y_0<y_1<\dots<y_{n-1}<y_n=1,
\]
where $y_1,\dots,y_{n-1}$ are the zeros of $P_n'$ and the weights $v_i>0$ are chosen so that $L_{n+1}$ is exact on $\Pi_{2n-1}$.
The endpoint weights satisfy $v_0=v_n=\frac{2}{n(n+1)}$.
Computation of nodes and weights is classical; see \cite{Gautschi2000,Gautschi2004}.\smallskip

\paragraph{\bf Gauss--Radau.}
For $n\ge 1$, the left and right Gauss--Radau rules with $(n+1)$ nodes are
\[
R^{\ell}_{n+1}[f] := \alpha_0 f(-1)+ \sum_{i=1}^n \alpha_i f(\xi_i),
\qquad
R^{r}_{n+1}[f] := \sum_{i=0}^{n-1}\beta_i f(\eta_i)+ \beta_n f(1),
\]
where the node sets include $-1$ and $1$, respectively, and each rule is exact on $\Pi_{2n}$.
We will only need them in Section~\ref{sec:even}.

\paragraph{\bf Gauss-type structural facts.}
We shall use the following standard facts in the classical Legendre case.
An $N$-node Gauss-type quadrature formula is one with algebraic degree of precision at least $2N-2$; the ordinary Gauss rule and the left and right Radau rules are the standard representatives, with the unrestricted Gauss rule having the still higher precision $2N-1$.
Such formulae have positive weights, and the nodes of any two members of the same Gauss-type family interlace.
The Lobatto rule is the corresponding principal formula obtained by prescribing both endpoints.
For the pair used below, $G_n$ and $L_{n+1}$ have the same algebraic degree of precision,
\[
   \operatorname{ADP}(G_n)=\operatorname{ADP}(L_{n+1})=2n-1.
\]
Thus the positivity of $G_n$ and $L_{n+1}$, the interlacing of the zeros of $P_n$ and $P'_n$, and the definite Peano-kernel comparison are not independent assumptions; they are consequences of this common Gauss-type structure.
We refer to Freud~\cite{Freud1971} for these classical facts.

\subsection{Affine changes of variables}

It is worth stating explicitly that all inequalities proved on the reference
interval $[-1,1]$ transfer to any finite interval $[a,b]$ by an affine change of
variables.
If
\[
T(x)=\frac{a+b}{2}+\frac{b-a}{2}\,x,
\]
then
\[
\int_a^b f(t)\,\dd t = \frac{b-a}{2}\int_{-1}^1 f(T(x))\,\dd x,
\]
and the Gauss, Gauss--Lobatto, and Gauss--Radau quadrature rules on $[a,b]$ are
obtained by applying the same affine rescaling to the nodes and multiplying the
weights by the factor $(b-a)/2$.

Concerning higher-order convexity, the behaviour under affine transformations
depends on the parity of the order.
If $f$ is $(m-1)$-convex on $[a,b]$ in the divided--difference sense and
$T(x)=ax+b$ with $a\neq 0$, then:
\begin{itemize}
\item if $m-1$ is \emph{odd} (that is, $m$ is even), the composition $f\circ T$
      is again $(m-1)$-convex for every $a\neq 0$;
\item if $m-1$ is \emph{even} and $a<0$, the composition $f\circ T$ is
      $(m-1)$-\emph{concave}.
\end{itemize}

This parity effect reflects the fact that, for even orders, the highest divided
difference (or, equivalently, the highest derivative when it exists) changes sign
under orientation reversal.
Since the present paper is concerned exclusively with odd-order convexity
$(2n-1)$, no distinction between increasing and decreasing affine maps is needed
in the sequel.

\section{Odd order: a refined Gauss--Lobatto inequality}\label{sec:odd}

The goal of this section is to explain why, for odd-order convexity, the integral is closer to the Gauss side than the
classical bracket suggests.
After recalling the known extremalities, we prove the refinement and discuss its immediate consequences.

\subsection{Classical extremalities}

The following theorem is a higher-order analogue of Hermite--Hadamard.
It is due to Bessenyei and Páles \cite{BessenyeiPales2002} in the framework of generalised convexity.

\begin{theorem}[Gauss and Lobatto bracket]\label{thm:classic-bracket}
Let $n\ge 1$ and let $f:[-1,1]\to\R$ be $(2n-1)$-convex.
Then
\[
G_n[f]\le \I[f]\le L_{n+1}[f].
\]
If $f$ is $(2n-1)$-concave, the inequalities reverse.
\end{theorem}

\begin{remark}
Extremal inequalities of Gauss--Legendre, Gauss--Lobatto, and Gauss--Radau type
for functions convex of higher order are not new in themselves.
In particular, it was shown by W\k{a}sowicz that for $(2n-1)$-convex functions
one has the Gauss--Lobatto extremality
\[
G_n[f]\le T[f]\le L_{n+1}[f]
\]
for every positive linear operator $T$ exact on $\Pi_{2n-1}$,
and that for $2n$-convex functions the corresponding extremal operators are
the left and right Gauss--Radau rules.
These results are proved without any differentiability assumptions,
using only higher-order convexity; see~\cite{Wasowicz2010Extremalities}.

The contribution of the present paper is not the extremality itself,
but the refined localisation of the integral inside the Gauss--Lobatto bracket
for odd orders, the resulting certified estimator, and the structural explanation
of the even-order obstruction via Radau quadrature kernels.
\end{remark}

\subsection{Peano kernels for the two error functionals}

Our refinement is most transparent in the language of Peano kernels.
The key observation is that the remainders of $G_n$ and $L_{n+1}$ admit Peano representations of the \emph{same} order.

\begin{definition}\label{def:error-functionals}
For $n\ge 1$ we define the \emph{Gauss} and \emph{Lobatto error functionals} by
\[
E^G_n[f] := \I[f]-G_n[f],
\qquad
E^L_n[f] := L_{n+1}[f]-\I[f].
\]
Both functionals annihilate $\Pi_{2n-1}$.
\end{definition}

\begin{lemma}[Peano kernel representations]\label{lem:peano}
Assume $f\in C^{2n}([-1,1])$.
Then there exist continuous kernels $K_G,K_L:[-1,1]\to\R$ such that
\[
E^G_n[f] = \int_{-1}^1 f^{(2n)}(t)K_G(t)\,\dd t,
\qquad
E^L_n[f] = \int_{-1}^1 f^{(2n)}(t)K_L(t)\,\dd t,
\]
and
\[
K_G(t)=\frac{1}{(2n-1)!}E^G_n\bigl[(\cdot-t)_+^{2n-1}\bigr],
\qquad
K_L(t)=\frac{1}{(2n-1)!}E^L_n\bigl[(\cdot-t)_+^{2n-1}\bigr].
\]
Moreover, $K_G$ and $K_L$ are non-negative.
\end{lemma}

\begin{proof}
This is the Peano kernel theorem with $m=2n$, applied to each error functional.
Continuity of $K_G$ and $K_L$ is automatic here because $E^G_n$ and $E^L_n$ are finite linear combinations of point evaluations:
for each fixed node $x$, the map $t\mapsto (x-t)_+^{2n-1}$ is continuous, hence so is
$t\mapsto E^{G/L}_n[(\cdot-t)_+^{2n-1}]$.

For the sign, fix $t\in[-1,1]$ and set $\phi_t(x):=(x-t)_+^{2n-1}$.
In the distributional sense one has
\[
D^{2n}\phi_t = (2n-1)!\,\delta_t\ge 0,
\]
so $\phi_t$ is $(2n-1)$-convex in the divided-difference sense (Definition~\ref{def:higherconv} and the standard equivalence with
non-negativity of the $2n$-th distributional derivative).
Therefore, Theorem~\ref{thm:classic-bracket} implies
$E^G_n[\phi_t]\ge 0$ and $E^L_n[\phi_t]\ge 0$, hence $K_G(t)\ge 0$ and $K_L(t)\ge 0$.

\end{proof}

\subsection{Kernel dominance and the refined bracket}

The refined inequality reduces to a pointwise comparison of Peano kernels.
Because this is the load-bearing comparison in the paper, we record precisely which external result is used.
In the Legendre setting this comparison is a Peano-kernel formulation of the standard Gauss-type structure discussed above: positivity, interlacing and the relevant definiteness properties stem from the common algebraic degree of precision of the principal pair.

\begin{lemma}[Kernel dominance]\label{lem:kernel-dominance}
For every $n\ge 1$ and every $t\in[-1,1]$,
\[
0\le K_G(t)\le K_L(t).
\]
Equivalently,
\[
E^G_n\bigl[(\cdot-t)_+^{2n-1}\bigr]\le E^L_n\bigl[(\cdot-t)_+^{2n-1}\bigr]
\qquad \text{for all }t\in[-1,1].
\]
\end{lemma}

\begin{proof}
We use the Peano-kernel estimate of Brass--Schmeisser~\cite[Theorem~6 and Remark~1]{BrassSchmeisser1981}; see also the discussion in Brass--Petras~\cite[Chapter~3]{BrassPetrasBook} and Szostok~\cite[Remark~3]{Szostok2024}.
Specialised to the Legendre weight and to odd degree $2n-1$, it says that the Peano kernel of the Gauss remainder is dominated in absolute value by the Peano kernel of the corresponding Gauss--Lobatto remainder.
Equivalently, for the two principal formulae of degree of precision $2n-1$,
\begin{equation}\label{eq:brass-schmeisser-kernel-use}
|K_G(t)|\le |K_L(t)|,\qquad -1\le t\le 1,
\end{equation}
where the kernels are taken in the orientations of Lemma~\ref{lem:peano}.
The proof in Brass--Schmeisser is a monospline/node-elimination argument for the Peano kernels; it is consistent with Freud's Gauss-type description in which the common algebraic degree of precision is the source of positivity and interlacing for the principal formulae.

We now verify the hypotheses in the present notation.
The Gauss rule $G_n$ has positive weights and degree of precision $2n-1$.
The Lobatto rule $L_{n+1}$ is the Legendre Gauss--Lobatto principal formula with the same degree of precision, positive endpoint weights, and positive interior weights.
Finally, Lemma~\ref{lem:peano} gives the definiteness signs in our orientation: $K_G\ge0$ for $\I-G_n$ and $K_L\ge0$ for $L_{n+1}-\I$.
Thus the absolute-value estimate \eqref{eq:brass-schmeisser-kernel-use} becomes $K_G(t)\le K_L(t)$ for all $t\in[-1,1]$.
The lower bound $K_G\ge0$ is already part of Lemma~\ref{lem:peano}.
\end{proof}

\begin{remark}The kernel dominance in Lemma~\ref{lem:kernel-dominance} is the only place where we use the special Legendre Gauss/Gauss--Lobatto principal-pair structure.
In this use of interlacing one should not separate the node picture from the algebraic degree of precision: within the Gauss-type family, positivity of the weights, interlacing of the nodes, and the inequality in Lemma~\ref{lem:kernel-dominance} are consequences of the common precision
$\operatorname{ADP}(G_n)=\operatorname{ADP}(L_{n+1})=2n-1$.
For other weights, such as Jacobi weights $w(x)=(1-x)^{\alpha}(1+x)^{\beta}$, analogous Peano kernels and positive Gauss/Lobatto rules exist, but the refined half-bracket would require the corresponding domination theorem for that weight.
We do not use, or claim, such an extension here.
\end{remark}

Before we introduce estimators, we pin down the precise refinement that drives the rest of the paper.

\begin{theorem}[Refined Gauss--Lobatto inequality]\label{thm:refined}
Let $n\ge 1$ and let $f:[-1,1]\to\R$ be $(2n-1)$-convex.
Then
\begin{equation}\label{eq:refined}
G_n[f]\le \I[f]\le \frac{G_n[f]+L_{n+1}[f]}{2}\le L_{n+1}[f].
\end{equation}
If $f$ is $(2n-1)$-concave, all inequalities reverse.
\end{theorem}

\begin{proof}
We rewrite the middle inequality as
\[
\I[f]\le \frac{G_n[f]+L_{n+1}[f]}{2}
\quad\Longleftrightarrow\quad
E^G_n[f]\le E^L_n[f].
\]
By Theorem~\ref{thm:popoviciu} with $m=2n$, we can represent
\[
f(x)=p(x)+\int_{-1}^1 (x-t)_+^{2n-1}\,\mu(\dd t),
\]
where $p\in\Pi_{2n-1}$ and $\mu$ is non-negative.
Since both remainders annihilate $\Pi_{2n-1}$, we obtain
\[
E^G_n[f] = \int_{-1}^1 E^G_n\bigl[(\cdot-t)_+^{2n-1}\bigr]\,\mu(\dd t),\qquad
E^L_n[f] = \int_{-1}^1 E^L_n\bigl[(\cdot-t)_+^{2n-1}\bigr]\,\mu(\dd t).
\]
Lemma~\ref{lem:kernel-dominance} implies that the first integrand is pointwise dominated by the second,
hence $E^G_n[f]\le E^L_n[f]$ and the desired refinement holds.
The concave case follows by applying the convex case to $-f$.
\end{proof}

\begin{remark}
For $n=1$, the refinement is a familiar strengthening of Hermite--Hadamard.
For $n=2$ and $n=3$ (the $3$-convex and $5$-convex cases), related inequalities between remainder terms were studied by
Komisarski and Wąsowicz \cite{KomisarskiWasowicz2017}.
\end{remark}

\section{Certified estimators and optimality}\label{sec:estimators}

Once the integral is confined to a shorter interval, the natural next step is to take its centre.
This section introduces the estimator $Q_n$, proves a sharp a posteriori bound, and records the elementary fact that $Q_n$ is the centre of the certified refined interval.

\subsection{A sharp certified one-point estimate}

The next proposition is the computational heart of the paper: it turns a qualitative inequality into a one-shot estimate
with a rigorous and (as we will see) sharp error bar.

\begin{proposition}[A certified estimator]\label{prop:Qn}
Let $n\ge 1$ and define
\[
Q_n[f]:=\frac{3}{4}G_n[f]+\frac{1}{4}L_{n+1}[f].
\]
If $f$ is $(2n-1)$-convex or $(2n-1)$-concave on $[-1,1]$, then
\begin{equation}\label{eq:Qn-bound}
\abs{\I[f]-Q_n[f]}\le \frac{1}{4}\abs{L_{n+1}[f]-G_n[f]}.
\end{equation}
\end{proposition}

\begin{proof}
Assume first that $f$ is $(2n-1)$-convex.
By Theorem~\ref{thm:refined}, $\I[f]\in\bracks{G_n[f],\,M_n[f]}$ where $M_n[f]=\frac{G_n[f]+L_{n+1}[f]}{2}$.
Moreover, $Q_n[f]=\frac{G_n[f]+M_n[f]}{2}$ is the midpoint of this interval.
Hence 
\[
    \abs{\I[f]-Q_n[f]}\le \frac{M_n[f]-G_n[f]}{2}=\frac{L_{n+1}[f]-G_n[f]}{4}.
\]
The concave case reduces to the convex case by replacing $f$ with $-f$:
$(2n-1)$-concavity is equivalent to $(2n-1)$-convexity of $-f$, and both
$\abs{\I[f]-Q_n[f]}$ and $\abs{L_{n+1}[f]-G_n[f]}$ are invariant under $f\mapsto -f$.
\end{proof}

\subsection{Optimality of the interval centre}

The coefficient $1/4$ first has an interval-theoretic interpretation.
Once Theorem~\ref{thm:refined} is known, the certified information about the integral is the refined interval
\[
\left[G_n[f],\,\frac{G_n[f]+L_{n+1}[f]}2\right]
\]
in the convex case, with the reversed interval in the concave case.
The best single affine number extracted from this certified interval is its centre.

\begin{proposition}[Best interval-based affine estimate]\label{prop:interval-centre}
Fix $n\ge 1$ and write
\[
A_\theta[f] := (1-\theta)G_n[f]+\theta L_{n+1}[f],
\qquad \theta\in\mathbb{R}.
\]
Among all affine combinations $A_\theta$, the choice $\theta=1/4$ uniquely minimises the worst-case distance from $A_\theta[f]$ to the certified refined interval.
The resulting worst-case radius is
\[
\frac14\abs{L_{n+1}[f]-G_n[f]}.
\]
\end{proposition}

\begin{proof}
It suffices to consider the convex case and write $D=L_{n+1}[f]-G_n[f]\ge 0$.
The refined interval is $[G_n[f],G_n[f]+D/2]$, while
$A_\theta[f]=G_n[f]+\theta D$.
The largest possible interval-based deviation is therefore
\[
D\max\{|\theta|,|1/2-\theta|\}.
\]
This expression is uniquely minimised at $\theta=1/4$, where its value is $D/4$.
The concave case is the same after replacing $f$ by $-f$.
\end{proof}

\subsection{Function-theoretic sharpness of the constant}

The preceding proposition is a decision statement about the centre of a certified interval.
The next proposition is different: it shows that the numerical constant in the error estimate cannot be improved on the cone of higher-order convex functions.

\begin{proposition}[Sharpness]\label{prop:sharpness}
The constant $\tfrac14$ in \eqref{eq:Qn-bound} is sharp.
More precisely, for every $n\ge 1$ there exists a sequence of $(2n-1)$-convex functions $(f_\varepsilon)_{\varepsilon>0}$ such that
\[
\lim_{\varepsilon\downarrow 0}\,
\frac{\abs{\I[f_\varepsilon]-Q_n[f_\varepsilon]}}{\abs{L_{n+1}[f_\varepsilon]-G_n[f_\varepsilon]}}
=\frac14.
\]
\end{proposition}

\begin{proof}
Fix $n$ and set $t_\varepsilon:=1-\varepsilon$.
Consider the truncated power
\[
f_\varepsilon(x):=(x-t_\varepsilon)_+^{2n-1},\qquad x\in[-1,1],
\]
which is $(2n-1)$-convex.
For $\varepsilon$ so small that $t_\varepsilon$ exceeds all interior nodes of both $G_n$ and $L_{n+1}$ (which is possible since all interior
Gauss and Lobatto nodes lie strictly in $(-1,1)$), we have
$f_\varepsilon(x_i)=0$ for every Gauss node $x_i$, and $f_\varepsilon(y_i)=0$ for every Lobatto node $y_i$ except $y_n=1$.
Therefore
\[
G_n[f_\varepsilon]=0,\qquad
L_{n+1}[f_\varepsilon]=v_n f_\varepsilon(1)=v_n\,\varepsilon^{2n-1},
\qquad v_n=\frac{2}{n(n+1)}.
\]
On the other hand,
\[
\I[f_\varepsilon]=\int_{t_\varepsilon}^1 (x-t_\varepsilon)^{2n-1}\,\dd x=\frac{\varepsilon^{2n}}{2n}.
\]
Hence $L_{n+1}[f_\varepsilon]-G_n[f_\varepsilon]=v_n\varepsilon^{2n-1}$ and
\[
Q_n[f_\varepsilon]=\frac14 L_{n+1}[f_\varepsilon]=\frac{v_n}{4}\,\varepsilon^{2n-1}.
\]
Consequently,
\[
\frac{\abs{\I[f_\varepsilon]-Q_n[f_\varepsilon]}}{\abs{L_{n+1}[f_\varepsilon]-G_n[f_\varepsilon]}}
=
\abs{\frac{\varepsilon}{2n v_n}-\frac14}
\;\xrightarrow[\varepsilon\downarrow 0]{}\;\frac14,
\]
as claimed.
\end{proof}

\begin{remark}[Two distinct optimality statements]
Proposition~\ref{prop:interval-centre} and Proposition~\ref{prop:sharpness} express two different senses of optimality.
The first is interval-theoretic: given only the certified inclusion supplied by Theorem~\ref{thm:refined}, the centre is the unique best affine decision.
The second is function-theoretic: even within the cone of $(2n-1)$-convex functions, the coefficient $1/4$ in the a posteriori bound is asymptotically attained by truncated powers concentrating at an endpoint.
Thus the propositions are complementary, not redundant.
\end{remark}

\subsection{Composite rules and stopping criteria}

The one-interval estimate is additive, which is exactly what is needed for certified composite integration.
Let $-1=a_0<a_1<\dots<a_m=1$ be a partition and let $Q_n^{\rm comp}$ be obtained by applying $Q_n$ on each subinterval.

\begin{proposition}[Composite certified bound]\label{prop:composite}
If $f$ is $(2n-1)$-convex or $(2n-1)$-concave on $[-1,1]$, then
\begin{equation}\label{eq:composite-abs}
\bigl|\I[f]-Q_n^{\rm comp}[f]\bigr|
\le \frac14\sum_{j=1}^m
\bigl|L_{n+1}^{[a_{j-1},a_j]}[f]-G_n^{[a_{j-1},a_j]}[f]\bigr|.
\end{equation}
More precisely, in the convex case set
\[
D_j:=L_{n+1}^{[a_{j-1},a_j]}[f]-G_n^{[a_{j-1},a_j]}[f]\ge0.
\]
Then on each subinterval
\[
-\frac14D_j
\le \I_{[a_{j-1},a_j]}[f]-Q_n^{[a_{j-1},a_j]}[f]
\le \frac14D_j.
\]
In the concave case the same statement holds with \(D_j\) replaced by \(|D_j|\).
\end{proposition}

\begin{proof}
Higher-order convexity is inherited by restriction to subintervals, and the quadrature rules are obtained from the reference rules by affine rescaling.
In the convex case, Theorem~\ref{thm:refined} gives on the $j$-th subinterval
\[
G_j\le \I_j\le G_j+\frac12D_j,
\qquad
Q_j=G_j+\frac14D_j,
\]
with the evident shorthand notation.
Hence \(\I_j-Q_j\in[-D_j/4,D_j/4]\).
The concave case follows by applying the same argument to \(-f\), which replaces \(D_j\) by \(|D_j|\).
Summing the local estimates and using the triangle inequality gives \eqref{eq:composite-abs}.
\end{proof}

\begin{remark}[Adaptive and uniform use]\label{rem:algorithm}
Formula \eqref{eq:composite-abs} gives a direct stopping criterion: refine until
\[
\sum_{j=1}^m \frac14
\bigl|L_{n+1}^{[a_{j-1},a_j]}[f]-G_n^{[a_{j-1},a_j]}[f]\bigr|
\le \varepsilon.
\]
The companion code implements both uniform refinement and greedy bisection of the subinterval with the largest local contribution.
This is a standard Peano-kernel certification paradigm; compare F\"orster's survey of stopping rules \cite{Foerster1993}.
The a posteriori estimate may also be viewed as an instance of the general linear-combination principle of Avdzhieva--Nikolov \cite[Theorem~1]{AvdzhievaNikolov2017}.
Although their applications concern definite quadrature formulae of fourth order, the theorem itself is formulated for a general pair of definite linear functionals; with the present pair $G_n,L_{n+1}$ it gives the same type of certified inequalities.
A cubature analogue of this principle is given by Nikolov--Nikolov \cite[Theorem~4]{NikolovNikolov2024}, where it is used to derive definiteness, monotonicity and a posteriori error estimates for modified product cubature formulae.
Thus the linear-combination mechanism is part of the established theory of definite formulae; the new point here is the refined Gauss--Lobatto half-bracket and the resulting sharp constant for the Legendre principal pair.
\end{remark}

\section{The smooth and analytic regime}\label{sec:smooth}

Although our main results are shape-theoretic, it is natural to ask how the certified estimator behaves on very smooth
integrands.  In that regime one can exploit classical error expansions and even design a degree-raising combination of
$G_n$ and $L_{n+1}$.

\subsection{Error constants and a degree-raising combination}

For smooth functions, the leading term of the error admits an explicit constant.
We recall one standard form for Gauss--Legendre; see, \emph{e.g.}, \cite{DavisRabinowitz}. In the next proposition, we explicitly derive the leading error coefficients.
\begin{proposition}\label{prop:error-constants}
Let $f\in C^{2n}([-1,1])$.
There exist points $\xi_G,\xi_L\in(-1,1)$ such that
\[
\I[f]-G_n[f] = c_n\, f^{(2n)}(\xi_G),\qquad
L_{n+1}[f]-\I[f] = \frac{n+1}{n}\,c_n\, f^{(2n)}(\xi_L),
\]
where
\[
c_n = \frac{2^{2n+1}(n!)^4}{(2n+1)\bigl((2n)!\bigr)^3}.
\]
Equivalently, the Lobatto coefficient is
\[
c_n^{L}:=
\frac{2^{2n+1}n^3(n+1)((n-1)!)^4}{(2n+1)\bigl((2n)!\bigr)^3}
=\frac{n+1}{n}c_n.
\]
\end{proposition}
\begin{proof}
By Lemma~\ref{lem:peano}, for $f\in C^{2n}([-1,1])$ one has
\[
\I[f]-G_n[f]=\int_{-1}^1 f^{(2n)}(t)\,K_G(t)\,\dd t,
\qquad
L_{n+1}[f]-\I[f]=\int_{-1}^1 f^{(2n)}(t)\,K_L(t)\,\dd t,
\]
where $K_G$ and $K_L$ are continuous and non-negative.
By the mean value theorem for integrals, there exist points $\xi_G,\xi_L\in(-1,1)$ such that
\[
\I[f]-G_n[f]= f^{(2n)}(\xi_G)\int_{-1}^1 K_G(t)\,\dd t,
\qquad
L_{n+1}[f]-\I[f]= f^{(2n)}(\xi_L)\int_{-1}^1 K_L(t)\,\dd t.
\]
Thus it suffices to identify the two kernel integrals.
For Gauss--Legendre, the classical error constant gives
\[
\int_{-1}^1 K_G(t)\,\dd t
=
\frac{2^{2n+1}(n!)^4}{(2n+1)\bigl((2n)!\bigr)^3}
=:c_n,
\]
see, for instance, Davis--Rabinowitz~\cite[Chapter~3]{DavisRabinowitz}.
For Gauss--Lobatto, the classical constant is
\[
\int_{-1}^1 K_L(t)\,\dd t
=
\frac{2^{2n+1}n^3(n+1)((n-1)!)^4}{(2n+1)\bigl((2n)!\bigr)^3}
=:c_n^L,
\]
see, for instance, Gautschi~\cite{Gautschi2000} or Davis--Rabinowitz~\cite[Section~2.7]{DavisRabinowitz}.
Since $(n!)^4=n^4((n-1)!)^4$, we have
\[
c_n^L=\frac{n+1}{n}c_n.
\]
Substituting these identities yields the stated mean value forms.
\end{proof}

Having a rigorous a~posteriori bound is essential for certification, but it is also useful to know what the estimator does in the
classical smooth regime.
The next identity isolates the leading error constant of $Q_n$ on the first non-exact monomial.
\begin{proposition}[Principal error constant of the certified estimator]\label{prop:Qn-asymptotic}
Let $Q_n=\tfrac34 G_n+\tfrac14 L_{n+1}$ as in Proposition~\ref{prop:Qn}.
Define the Peano kernel
\[
K_Q(t):=\frac34 K_G(t)-\frac14 K_L(t),\qquad t\in[-1,1],
\]
where $K_G,K_L$ are the kernels from Lemma~\ref{lem:peano}.
Then for every $f\in C^{2n}([-1,1])$ one has the exact representation
\[
\I[f]-Q_n[f]=\int_{-1}^1 f^{(2n)}(t)\,K_Q(t)\,\dd t.
\]
Moreover,
\begin{equation}\label{eq:Qn-leading-constant}
\int_{-1}^1 K_Q(t)\,\dd t
=
c_n\,\frac{2n-1}{4n},
\end{equation}
so that for the monomial $f(x)=x^{2n}$,
\[
\I[x^{2n}]-Q_n[x^{2n}]=(2n)!\,c_n\,\frac{2n-1}{4n}.
\]
In particular, $Q_n$ has the same formal order as $G_n$ and $L_{n+1}$, but its principal error constant is smaller than
that of $G_n$ by the factor $(2n-1)/(4n)$.
\end{proposition}

\begin{proof}
The kernel identity is immediate from Lemma~\ref{lem:peano} and linearity:
\[
\I[f]-Q_n[f]
=\frac34\bigl(\I[f]-G_n[f]\bigr)-\frac14\bigl(L_{n+1}[f]-\I[f]\bigr)
=\int_{-1}^1 f^{(2n)}(t)\Bigl(\frac34 K_G(t)-\frac14 K_L(t)\Bigr)\,\dd t.
\]
To compute the integral of $K_Q$, apply the preceding identity to $f(x)=x^{2n}$.
Since $f^{(2n)}\equiv (2n)!$ is constant, one obtains
\[
\I[x^{2n}]-Q_n[x^{2n}]=(2n)!\int_{-1}^1 K_Q(t)\,\dd t.
\]
On the other hand, Proposition~\ref{prop:error-constants} yields
$\I[x^{2n}]-G_n[x^{2n}]=(2n)!\,c_n$ and $L_{n+1}[x^{2n}]-\I[x^{2n}]=\frac{n+1}{n}(2n)!\,c_n$.
Substituting into the definition of $Q_n$ gives
\[
\I[x^{2n}]-Q_n[x^{2n}]
=\frac34(2n)!\,c_n-\frac14\frac{n+1}{n}(2n)!\,c_n
=(2n)!\,c_n\,\frac{2n-1}{4n},
\]
which implies \eqref{eq:Qn-leading-constant}.
\end{proof}

\begin{proposition}[A degree-$2n+1$ combination]\label{prop:Pn}
Define
\[
P_n[f] := \frac{n+1}{2n+1}G_n[f]+\frac{n}{2n+1}L_{n+1}[f].
\]
Then $P_n$ is exact on $\Pi_{2n+1}$.
\end{proposition}

\begin{proof}
Both $G_n$ and $L_{n+1}$ are exact on $\Pi_{2n-1}$, hence so is $P_n$.
Because the rules are symmetric, they are also exact for every odd polynomial, in particular for $x^{2n+1}$.
It therefore suffices to check exactness on $x^{2n}$.
Writing $\I[x^{2n}]-G_n[x^{2n}]=\gamma$ and $L_{n+1}[x^{2n}]-\I[x^{2n}]=\frac{n+1}{n}\gamma$ (Proposition~\ref{prop:error-constants}),
we find
\[
\I[x^{2n}]-P_n[x^{2n}]
= \frac{n+1}{2n+1}\gamma-\frac{n}{2n+1}\cdot\frac{n+1}{n}\gamma
=0.
\]
\end{proof}

\begin{remark}
The rule $P_n$ is the standard degree-raising cancellation obtained by combining two symmetric rules of the same precision so that their leading $x^{2n}$ errors cancel.
It can be viewed as a very small ``extrapolation'' based solely on $G_n$ and $L_{n+1}$.
Unlike $Q_n$, it is not certified on the class of $(2n-1)$-convex functions because it is not constrained to the refined interval.
\end{remark}

\subsection{Analytic functions and complex error kernels}

For analytic integrands, quadrature errors exhibit geometric convergence governed by complex Peano kernels.
A typical bound replaces the real integral of $K_Q$ by a contour integral over a Bernstein ellipse or another domain of analyticity, so the error is controlled by the maximum of $|f|$ on the contour and by the size of the corresponding complex kernel.
Hunter and Nikolov \cite{HunterNikolov1999} study this error term for symmetric Gauss--Lobatto formulas and analyse the location of maxima of the associated complex kernel.
Such analytic estimates complement the present work: they can be extremely sharp when a trustworthy analyticity domain is available, whereas our certified inequalities are shape-driven and hold without analyticity.
The a posteriori estimates in the preceding section should, moreover, be read in the light of Avdzhieva--Nikolov \cite[Theorem~1]{AvdzhievaNikolov2017}: the general definite-functional principle already explains how appropriate linear combinations yield such estimates.

\section{Even order: Radau quadrature kernels and an obstruction}\label{sec:even}

Odd-order convexity interacts with the symmetric Gauss/Gauss--Lobatto pair.
Even order is different: the extremal rules are one-sided Radau rules, and the symmetric midpoint of those two rules cannot have a fixed sign on the cone.

\subsection{The Radau bracket for \texorpdfstring{$2n$}{2n}-convexity}

\begin{theorem}[Radau bracket]\label{thm:radau-bracket}
Let $n\ge 1$ and let $f:[-1,1]\to\R$ be $2n$-convex.
Then
\[
R^{\ell}_{n+1}[f]\le \I[f]\le R^{r}_{n+1}[f].
\]
If $f$ is $2n$-concave, the inequalities reverse.
\end{theorem}

\begin{proof}[Proof sketch]
The functionals
\[
E^{\ell}[f]=\I[f]-R^{\ell}_{n+1}[f],
\qquad
E^{r}[f]=R^{r}_{n+1}[f]-\I[f]
\]
annihilate $\Pi_{2n}$, because the Radau rules are exact on $\Pi_{2n}$.
Their Peano kernels of order $2n+1$ are non-negative; equivalently the Radau rules are definite of this order.
By the reduction principle in Remark~\ref{rem:reduction}, this gives the asserted bracket for $2n$-convex functions, and the concave case follows from $f\mapsto -f$.
See W\k{a}sowicz~\cite{Wasowicz2010Extremalities} for the corresponding general extremality theorem, and Davis--Rabinowitz~\cite[Section~2.7]{DavisRabinowitz} or Brass--Petras~\cite[Chapter~3]{BrassPetrasBook} for the quadrature background.
\end{proof}

\subsection{Why the Radau midpoint cannot refine the bracket}

\begin{definition}[Radau quadrature kernel]\label{def:radau-kernel}
For $n\ge 1$, define
\[
M^R_n[f]:=\frac12\bigl(R^{\ell}_{n+1}[f]+R^{r}_{n+1}[f]\bigr).
\]
Since each Radau rule is exact on $\Pi_{2n}$, the Peano kernel of $M^R_n-\I$ of order $2n+1$ is
\[
K^R_n(t):=\frac{1}{(2n)!}\,(M^R_n-\I)\bigl[(\cdot-t)_+^{2n}\bigr].
\]
Thus $(M^R_n-\I)[f]=\int_{-1}^1 f^{(2n+1)}(t)K^R_n(t)\,\mathrm{d}t$ whenever $f\in C^{2n+1}$.
\end{definition}

\begin{theorem}[Oddness and sign change]\label{thm:radau-odd}
For every $n\ge 1$, the Radau quadrature kernel $K^R_n$ is odd and non-zero.
Consequently it changes sign on $(-1,1)$.
\end{theorem}

\begin{proof}
Write $M=M^R_n$ and let $\rho(x)=-x$.
The reflection $\rho$ maps the nodes of $R^{\ell}_{n+1}$ to the nodes of $R^{r}_{n+1}$ and preserves the corresponding weights, since the two Radau rules are obtained from each other by this reflection.
Thus $R^{\ell}_{n+1}[f\circ\rho]=R^{r}_{n+1}[f]$ and $R^{r}_{n+1}[f\circ\rho]=R^{\ell}_{n+1}[f]$.
Consequently $M[f\circ\rho]=M[f]$, and $\I[f\circ\rho]=\I[f]$ as well.
For $\phi_t(x)=(x-t)_+^{2n}$, reflection gives
\[
(M-\I)[\phi_t]=(M-\I)\bigl[(\cdot+t)_-^{2n}\bigr],
\]
where $u_-:=\max\{-u,0\}$.
Since
\[
(\cdot+t)_+^{2n}+(\cdot+t)_-^{2n}=(\cdot+t)^{2n}\in\Pi_{2n},
\]
exactness of $M$ on $\Pi_{2n}$ implies
\[
(M-\I)\bigl[(\cdot+t)_+^{2n}\bigr]
=-(M-\I)\bigl[(\cdot+t)_-^{2n}\bigr].
\]
Therefore $K^R_n(-t)=-K^R_n(t)$.

It remains only to see that the kernel is not identically zero.
Let $\beta_n=2/(n+1)^2$ be the weight of the right Radau rule at $1$.
Choose $t<1$ larger than all other Radau nodes.
Then only the endpoint $1$ contributes to $M[(\cdot-t)_+^{2n}]$, so
\[
(M-\I)[(\cdot-t)_+^{2n}]
=(1-t)^{2n}\left(\frac{\beta_n}{2}-\frac{1-t}{2n+1}\right),
\]
which is positive for $t$ sufficiently close to $1$.
By oddness, the kernel is negative near $-1$.
\end{proof}

\begin{corollary}[No universal refinement in even order]\label{cor:no-even-refinement}
Fix $n\ge 1$.
There is no sign choice $\sigma\in\{-1,+1\}$ such that
\[
\sigma\bigl(M^R_n[f]-\I[f]\bigr)\ge 0
\quad\text{for every $2n$-convex }f.
\]
\end{corollary}

\begin{proof}
By the representation theorem, it suffices to test the extremal functions $(\cdot-t)_+^{2n}$.
For these functions the value of $M^R_n-\I$ is $(2n)!K^R_n(t)$, which changes sign by Theorem~\ref{thm:radau-odd}.
\end{proof}

\begin{remark}
This obstruction is structural rather than numerical.
Equivalently, one may phrase it as the fact that a definite quadrature formula of odd Peano order is necessarily one-sided and cannot be symmetric.
For the lowest orders, related examples were analysed by Komisarski--W\k{a}sowicz~\cite{KomisarskiWasowicz2017}; the kernel argument above gives the obstruction for all $n$.
\end{remark}

\section{Numerical experiments}\label{sec:numerics}

The refined estimate is a certification result, not a claim that a low-order formula beats unconstrained high-order quadrature on every smooth test function when the exact answer is known.
The relevant numerical question is different: does the adaptive composite estimate from Proposition~\ref{prop:composite} give a practical, computable certificate on integrands that arise naturally in the applications mentioned in the introduction?
We therefore report tests for spline-type truncated powers and for Stieltjes kernels, and we include an explicit cost comparison with adaptive Simpson quadrature.

All experiments were generated by the companion code (version \codeVersion{}):
\begin{center}
\codeRepo{}
\end{center}
The repository contains the reusable Python package, unit tests, and notebooks for regenerating the tables and kernel plots.
Function-evaluation counts below are counts of distinct abscissae after caching reused endpoints and previously evaluated subdivision nodes.

\subsection{A spline-basis diagnostic}

Truncated powers are the basic local building blocks in spline approximation:
\[
\psi_{c,r}(x)=(x-c)_+^r .
\]
For the first genuinely new case in this paper, $n=4$, the function
\[
\psi_{0.37,7}(x)=(x-0.37)_+^7,
\qquad 0\le x\le 1,
\]
is $7$-convex but is not $C^8$.
Thus the usual smooth Gauss error constant involving a bounded eighth derivative is not available; distributionally, the eighth derivative is a positive point mass at the knot.
At the same time this example should not be overinterpreted as a difficult integration problem: if the knot is known, then
\[
\int_0^1 (x-0.37)_+^7\,\mathrm{d}x
 =\int_{0.37}^1 (x-0.37)^7\,\mathrm{d}x
 =\frac{0.63^8}{8},
\]
and the last polynomial integral is evaluated exactly by a four-point Gauss rule after splitting at the knot.
We therefore use this example only as a diagnostic for the shape certificate.
The greedy certified $Q_4$ integrator reaches tolerance $10^{-8}$ using two subintervals and $23$ distinct function evaluations; the certified bound is $2.382\times 10^{-9}$ and the actual error is $1.614\times10^{-9}$.
For tolerance $10^{-10}$ it uses three subintervals and $37$ evaluations, with certified bound $8.177\times10^{-11}$ and actual error $6.553\times10^{-11}$.

\subsection{Near-pole Stieltjes kernels and adaptive cost}

A more informative class for certification is provided by Stieltjes-type kernels
\[
f_\delta(x)=\frac{1}{x+\delta},
\qquad 0\le x\le 1,
\qquad \delta>0.
\]
These functions satisfy $f_\delta^{(8)}(x)>0$, so they are $7$-convex and the $Q_4$ certificate applies on every subinterval.
The exact integral is
\[
\int_0^1 \frac{\mathrm{d}x}{x+\delta}=\log\frac{1+\delta}{\delta}.
\]
When $\delta$ is small, the pole at $-\delta$ is close to the interval and global smoothness constants are very pessimistic, even though the shape condition remains simple.

Table~\ref{tab:stieltjes-adaptive} compares the greedy certified $Q_4$ strategy with the classical adaptive Simpson routine of Lyness type; see Lyness~\cite{Lyness1969} and the survey of Gonnet~\cite{Gonnet2012}.
Both methods were run with absolute tolerance $10^{-8}$.
For $Q_4$ the error column labelled ``cert./true'' gives the rigorous bound from Proposition~\ref{prop:composite} followed by the actual error against the closed form integral.
For Simpson the corresponding pair is the usual Simpson error indicator followed by the actual error; the Simpson indicator is included as a standard adaptive comparison, but it is not a higher-order-convexity certificate.

\begin{table}[H]
\centering
\small
\begin{tabular}{r r r c r r c}
\toprule
& \multicolumn{3}{c}{certified greedy $Q_4$} & \multicolumn{3}{c}{adaptive Simpson} \\
\cmidrule(lr){2-4}\cmidrule(lr){5-7}
$\delta$ & subints & evals & cert./true error & subints & evals & indicator/true error \\
\midrule
$10^{-2}$ & 16 & 219 & $9.132\!\times\!10^{-9}/7.067\!\times\!10^{-9}$ & 189  & 757   & $3.388\!\times\!10^{-9}/5.261\!\times\!10^{-12}$ \\
$10^{-4}$ & 37 & 513 & $7.890\!\times\!10^{-9}/6.109\!\times\!10^{-9}$ & 789  & 3157  & $3.326\!\times\!10^{-9}/5.073\!\times\!10^{-12}$ \\
$10^{-6}$ & 58 & 807 & $7.776\!\times\!10^{-9}/6.023\!\times\!10^{-9}$ & 2664 & 10657 & $3.334\!\times\!10^{-9}/5.087\!\times\!10^{-12}$ \\
\bottomrule
\end{tabular}
\caption{Adaptive integration of $f_\delta(x)=1/(x+\delta)$ on $[0,1]$ at tolerance $10^{-8}$.
The $Q_4$ bound is a rigorous shape-based certificate for $7$-convex integrands; the Simpson indicator is the standard embedded adaptive Simpson estimate and is shown only for cost comparison.}
\label{tab:stieltjes-adaptive}
\end{table}

The table shows the expected behaviour.
As the pole approaches the interval, both algorithms refine near the left endpoint.
The shape-based $Q_4$ certificate remains within a small factor of the true error and needs substantially fewer function evaluations than this standard adaptive Simpson implementation for the three reported values of $\delta$.
Adaptive Simpson happens to deliver a much smaller true error in this example, but it does so with a much larger evaluation count and without using the higher-order convexity certificate.

\subsection{Kernel plots}

To complement the theory, Figure~\ref{fig:peano-kernels} plots the Peano kernels $K_G$ and $K_L$ and their half-difference
for $n=4$, illustrating the dominance $K_G\le K_L$.
Figure~\ref{fig:radau-kernel} plots the Radau quadrature kernel $K^R_2$, illustrating its oddness and sign change.

\begin{figure}[H]
\centering
\includegraphics[width=0.76\linewidth]{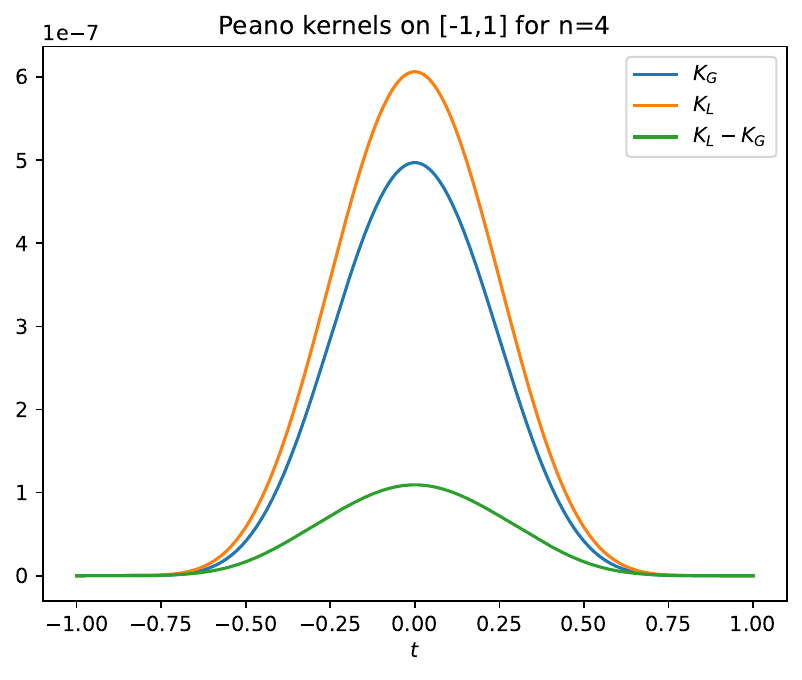}
\caption{Peano kernels on $[-1,1]$ for $n=4$: the Gauss kernel $K_G$, the Lobatto kernel $K_L$, and their half-difference.
The pointwise dominance $0\le K_G\le K_L$ is visible.}
\label{fig:peano-kernels}
\end{figure}

\begin{figure}[H]
\centering
\includegraphics[width=0.76\linewidth]{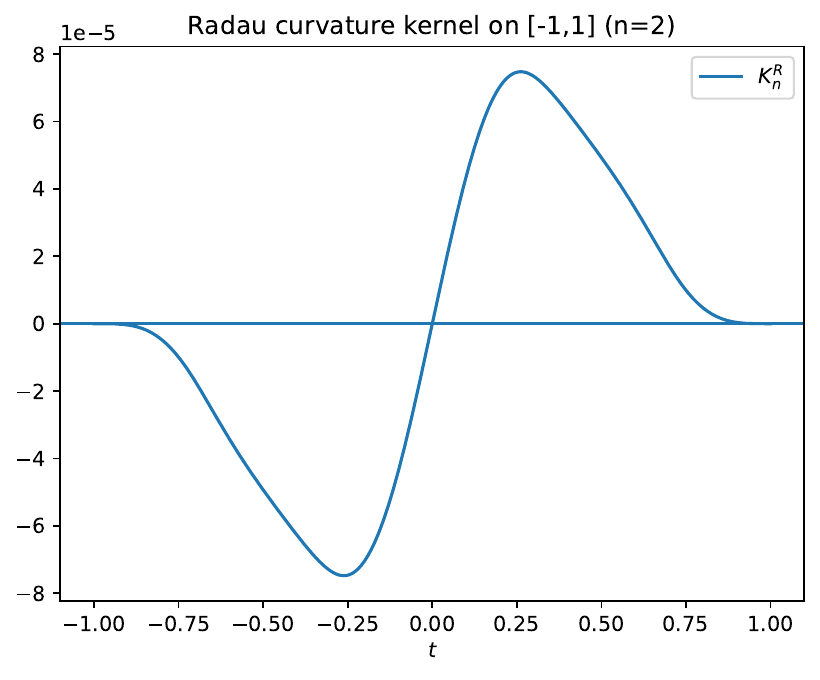}
\caption{The Radau quadrature kernel $K^R_n$ for $n=2$.
The kernel is odd and therefore must change sign, which obstructs a one-sided refinement of the Radau bracket in even order.}
\label{fig:radau-kernel}
\end{figure}

\subsection*{Acknowledgements}
The authors thank the anonymous referees for their careful reading of the manuscript and for their constructive comments and suggestions, which helped improve the exposition and strengthen several arguments.

\bibliographystyle{plain}

\begin{thebibliography}{99}
\bibitem{AvdzhievaNikolov2017}
A.~Avdzhieva and G.~Nikolov.
\newblock Asymptotically optimal definite quadrature formulae of fourth order.
\newblock {\em Journal of Computational and Applied Mathematics}, 311:565--582, 2017.
\newblock \url{https://doi.org/10.1016/j.cam.2016.08.022}

\bibitem{BalabdaouiWellner2007}
F.~Balabdaoui and J.~A. Wellner.
\newblock Estimation of a $k$-monotone density: limit distribution theory and the spline connection.
\newblock {\em The Annals of Statistics}, 35(6):2536--2564, 2007.
\newblock \url{https://doi.org/10.1214/009053607000000262}

\bibitem{BessenyeiPales2002}
M.~Bessenyei and Z.~Páles.
\newblock Higher-order convexity and {H}ermite--{H}adamard-type inequalities.
\newblock {\em Publ. Math. Debrecen}, 61(3--4):623--643, 2002.
\newblock \url{https://doi.org/10.5486/PMD.2002.2706}

\bibitem{BessenyeiPales2010}
M.~Bessenyei and Z.~Páles.
\newblock Characterization of higher-order monotonicity via integral inequalities.
\newblock {\em Proc. Roy. Soc. Edinburgh Sect. A}, 140(4):723--736, 2010.
\newblock \url{https://doi.org/10.1017/S0308210509001188}

\bibitem{BoasWidder1940}
R.~P. Boas Jr. and D.~V. Widder.
\newblock Functions with positive differences.
\newblock {\em Duke Math. J.}, 7(1):496--503, 1940.
\newblock \url{https://doi.org/10.1215/S0012-7094-40-00729-3}

\bibitem{BrassPetrasBook}
Helmut Brass and Knut Petras.
\newblock {\em Quadrature theory}, volume 178 of {\em Mathematical Surveys and
  Monographs}.
\newblock American Mathematical Society, Providence, RI, 2011.
\newblock The theory of numerical integration on a compact interval.


\bibitem{BrassSchmeisser1981}
H.~Brass and G.~Schmeisser.
\newblock Error estimates for interpolatory quadrature formulae.
\newblock {\em Numerische Mathematik}, 37:371--386, 1981.
\newblock \url{https://doi.org/10.1007/BF01400316}

\bibitem{DavisRabinowitz}
P.~J. Davis and P.~Rabinowitz.
\newblock {\em Methods of Numerical Integration}.
\newblock Academic Press, second edition, 1984.


\bibitem{deBoor2001}
C.~de~Boor.
\newblock {\em A Practical Guide to Splines}.
\newblock Applied Mathematical Sciences, vol.~27. Springer, revised edition, 2001.
\newblock \url{https://doi.org/10.1007/978-1-4612-6333-3}

\bibitem{EsmaeiliMilovanovic2014}
S.~Esmaeili and G.~V. Milovanovi\'c.
\newblock A {G}auss--{L}obatto quadrature-based scheme for solving fractional
  differential equations.
\newblock {\em Fractional Calculus and Applied Analysis}, 17(4):1075--1099,
  2014.
\newblock \url{https://doi.org/10.2478/s13540-014-0215-z}

\bibitem{Foerster1993}
K.-J.~F\"orster.
\newblock Survey on stopping rules in quadrature based on Peano kernel methods.
\newblock {\em Supplemento ai Rendiconti del Circolo Matematico di Palermo, Serie II}, 33:311--330, 1993.

\bibitem{Gautschi2000}
W.~Gautschi.
\newblock High-order {G}auss--{L}obatto formulae.
\newblock {\em Numerical Algorithms}, 25:213--222, 2000.
\newblock \url{https://doi.org/10.1023/A:1016689830453}

\bibitem{Gautschi2004}
W.~Gautschi.
\newblock On generating generalized {G}auss--{R}adau and {G}auss--{L}obatto
  formulae.
\newblock {\em BIT Numerical Mathematics}, 44(4):711--720, 2004.
\newblock \url{https://doi.org/10.1007/s10543-004-3812-0}



\bibitem{Freud1971}
G.~Freud.
\newblock {\em Orthogonal Polynomials}.
\newblock Pergamon Press, Oxford, 1971.

\bibitem{Gonnet2012}
P.~Gonnet.
\newblock A review of error estimation in adaptive quadrature.
\newblock {\em ACM Computing Surveys}, 44(4):22, 2012.
\newblock \url{https://doi.org/10.1145/2333112.2333117}

\bibitem{GroeneboomJongbloedWellner2001}
P.~Groeneboom, G.~Jongbloed, and J.~A. Wellner.
\newblock Estimation of a convex function: characterizations and asymptotic theory.
\newblock {\em The Annals of Statistics}, 29(6):1653--1698, 2001.
\newblock \url{https://doi.org/10.1214/aos/1015345958}

\bibitem{JagelsReichel2007}
C.~Jagels and L.~Reichel.
\newblock Szeg\H{o}--{L}obatto quadrature rules.
\newblock {\em Journal of Computational and Applied Mathematics}, 200:116--126, 2007.
\newblock \url{https://doi.org/10.1016/j.cam.2005.12.009}


\bibitem{Hopf1926}
E.~Hopf,
\emph{Elementare Bemerkungen \"uber die L\"osungen partieller Differentialgleichungen zweiter Ordnung vom elliptischen Typus},
Sitzungsber.\ Preuss.\ Akad.\ Wiss.\ Phys.-Math.\ Kl.\ (1926), 147--152.

\bibitem{HunterNikolov1999}
D.~B. Hunter and G.~Nikolov.
\newblock Error term of symmetric {G}auss--{L}obatto quadrature formulas for
  analytic functions.
\newblock {\em Mathematics of Computation}, 69(229):269--282, 2000.
\newblock \url{https://doi.org/10.1090/S0025-5718-99-01078-9}


\bibitem{KarlinStudden1966}
S.~Karlin and W.~J. Studden.
\newblock {\em Tchebycheff Systems: With Applications in Analysis and Statistics}.
\newblock Pure and Applied Mathematics, vol.~15. Interscience, 1966.

\bibitem{KomisarskiWasowicz2017}
A.~Komisarski and S.~W\k{a}sowicz.
\newblock Inequalities between remainders of quadratures.
\newblock {\em Aequationes Mathematicae}, 91:1103--1114, 2017.
\newblock \url{https://doi.org/10.1007/s00010-017-0505-8}

\bibitem{Lyness1969}
J.~N. Lyness.
\newblock Notes on the adaptive Simpson quadrature routine.
\newblock {\em Journal of the ACM}, 16(3):483--495, 1969.
\newblock \url{https://doi.org/10.1145/321526.321537}

\bibitem{NikolovNikolov2024}
G.~Nikolov and P.~Nikolov.
\newblock Modified trapezoidal product cubature rules: definiteness, monotonicity, and a posteriori error estimates.
\newblock {\em Mathematics}, 12(23):3783, 2024.
\newblock \url{https://doi.org/10.3390/math12233783}

\bibitem{NiculescuPersson}
C.~P. Niculescu and L.-E. Persson.
\newblock {\em Convex Functions and Their Applications: A Contemporary Approach}.
\newblock CMS Books in Mathematics. Springer, second edition, 2018.

\bibitem{Popoviciu1944}
T.~Popoviciu,
\emph{Les fonctions convexes},
Actualit\'es Sci.\ Ind.\ 992,
Hermann, Paris, 1944.

\bibitem{RobertsVarberg}
A.~W. Roberts and D.~E. Varberg.
\newblock {\em Convex Functions}.
\newblock Academic Press, New York, 1973.

\bibitem{Szego1939}
G.~Szeg{\H o},
\emph{Orthogonal Polynomials},
Amer.\ Math.\ Soc.\ Colloq.\ Publ.\ Vol.~23,
American Mathematical Society, 1939.

\bibitem{Szostok2024}
T.~Szostok.
\newblock A generalization of a theorem of Brass and Schmeisser.
\newblock {\em Numer. Math.}, 156(5):1915--1925, 2024.
\newblock \url{https://doi.org/10.1007/s00211-024-01434-7}

\bibitem{Wasowicz2010Extremalities}
S.~W\k{a}sowicz,
\emph{On some extremalities in the approximate integration},
Math.\ Inequal.\ Appl.\ \textbf{13} (2010), no.~1, 165--174.
\newblock \url{https://doi.org/10.7153/mia-13-13}

\bibitem{Was20}
S.~W\k{a}sowicz,
\newblock On a certain adaptive method of approximate integration and its stopping criterion, \emph{Aequationes Math.}\ 94(5):887--898, 2020.
\newblock \url{https://doi.org/10.1007/s00010-020-00719-0}

\bibitem{Was26_arxiv}
S.~W\k{a}sowicz.
\newblock Adaptive integration of 5-convex and 5-concave functions.
\newblock Preprint, arXiv:2601.16796, 2026. \url{https://arxiv.org/abs/2601.16796}


\bibitem{Williams2006}
P.~Williams.
\newblock Hermite--{L}obatto and {G}auss--{L}obatto methods for optimal control
  problems.
\newblock {\em ANZIAM Journal}, 47:C792--C805, 2006.

\end{thebibliography}

\end{document}